\documentclass[preprint,12pt]{elsarticle} 
\usepackage{amsthm,amsfonts,amssymb,amscd,amsmath,enumerate,verbatim,calc,graphicx,geometry} 
\usepackage[all]{xy} 
\usepackage{mathtools}
\usepackage{tikz}
\usepackage{pgfplots}
 \usepackage{tikz-3dplot}

\usepackage{mathtools}
\allowdisplaybreaks

\newtheorem{theorem}{Theorem}[section] 
\newtheorem{lemma}[theorem]{Lemma} 
\newtheorem{proposition}[theorem]{Proposition} 
\newtheorem{corollary}[theorem]{Corollary} 
\theoremstyle{definition} 
\newtheorem{definition}[theorem]{Definition}

\newtheorem{remark}[theorem]{Remark} 
\newtheorem{example}[theorem]{Example} 

\journal{ } 
\begin{document} 
\begin{frontmatter} 
\title{Discrete homotopic distance between Lipschitz maps} 
\author[]{Elahe~Hoseinzadeh\corref{cor2}}
\ead{hoseinzadeh.e@gmail.com}
\author[]{Hanieh~Mirebrahimi\corref{cor1}}
\ead{h_mirebrahimi@um.ac.ir}
\author[]{Hamid~Torabi\corref{cor3}}
\ead{h.torabi@ferdowsi.um.ac.ir}
\author[]{Ameneh~Babaee\corref{cor4}}
\ead{ambabaee057@gmail.com}
\address{Department of Pure Mathematics, Ferdowsi University of Mashhad,\\ 
P.O.Box 1159-91775, Mashhad, Iran.} 
\cortext[cor1]{Corresponding author} 
\begin{abstract} 
In this paper, we investigate a discrete version of the homotopic distance between two $s$-Lipschitz maps for $s \geq 0$. This distance is defined by specifying a step length $r$ to which some homotopy relation corresponds. In spaces with a significant number of holes, where no continuous homotopy exist and the homotopic distance equals infinite, the discrete homotopic distance provides a meaningful classification by effectively ignoring smaller holes. We show that the discrete homotopic distance $D_r$ generalizes key concepts such as the discrete Lusternik-Schnirelmann category $\text{cat}_r$ and the discrete topological complexity $\text{TC}_r$. Furthermore, we prove that $D_r$ is invariant under discrete homotopy relations. This approach offers a flexible framework for classifying $s$-Lipschitz maps, loops, and paths based on the choice of $r$.
\end{abstract} 

\end{frontmatter} 
\section{Introduction} 

Mac\'{i}as-Virg\'{o}s and others \cite{Mac'ias} introduced the concept of homotopic distance as a measure of distance between two  maps,  based on the homotopy criterion. Note that by "map" we mean a continuous map. 
This measure can help to classify and estimate maps, particularly those for which the distance is zero.
The homotopic distance between two maps is defined as the minimum cardinality minus $1$ of any open covering of the domain such that the restricted maps on each element of the covering are homotopic. This concept is homotopy invariant and, in certain special cases, coincides with other topological invariants such as topological complexity and Lusternik-Schnirelmann category. This invariant serves as a useful tool for solving open problems in topological complexity by providing a unified framework to study homotopy theory. The systematic unification of proofs for various properties of these invariants not only offers a more efficient approach but also facilitates the exploration of new insights \cite{Mac'ias}.
In the definition of homotopic distance, the deformation of the restricted maps is continuous. However, there are many cases where continuity is not necessary; Especially if a homotopy process cannot be found. Therefore, the homotopic distance may not provide an appropriate solution for situations where no continuous deformation exists. 
Many spaces have structures with a large number of holes, making it challenging to create a continuous mapping.
In these spaces, one may propose to ignore small holes. This solution has already been widely used by robots moving  discretely and point by point.
While some robots are designed to move continuously, others-such as linear industrial robots, automated assembly robots, research robots, and service robots operating in public spaces-rely on discrete motions. These robots use discrete movements to avoid obstacles and prevent collisions. For robots that move discretely, the topology of the environment is often modelled point-to-point, typically as a grid or graph of nodes and connections \cite{Borat1,Vergilia, Fern}. 
 Homotopy is a kind of motion and holes are obstacles for this movement. Accordingly, discretization seems to be a suitable solution for these spaces. Digital modes are an attempt to discretize the concepts.
Borat \cite{Borat1} introduced the concept of digital homotopic distance, particularly used in digital images and digital functions. Then the distance between simplicial maps were defined and studied in \cite{Borat2}.

In this paper, we introduce a new version of homotopic distance to offer innovative perspectives on topological structures in applications such as motion planning algorithms, image analysis, and data science, where discrete deformations are utilized. 
In discrete homotopic distance, the structure of the space and the map remain unchanged, but the deformation and homotopy between maps are discretized by using Lipschitz functions. This discrete deformation allows maps to traverse holes, whereas in continuous homotopy, even a small hole can be a significant obstacle to be deformed. In essence, discrete homotopic distance enables maps to bypass holes by taking larger steps. This kind of homotopy was defined in \cite{hasanzadeh} as a generalization of $r$-homotopy.
Barcelo defined $r$-homotopy for Lipschitz maps as the combination of continuous paths within a discrete framework \cite{Barcelo}. 

The paper is organized as follows: In Section 2, we first recall the concept of discrete homotopy, termed $(s, r)$-homotopy, to introduce the discrete homotopic distance denoted by $D_r$. This value depends inversely on $r$; Naturally, taking longer steps leads to arrive faster. As an example, we present two non-homotopic Lipschitz functions whose discrete homotopic distance vanishes for sufficiently large $r$.

In Section 3, we investigate the relationship between the discrete homotopic distance and some discrete categories defined for maps, such as the discrete Lusternik-Schnirelmann category and discrete topological complexity being previously defined in \cite{hasanzadeh}.
It is important to note that the homotopic distance encompasses well-known invariants, including the Lusternik-Schnirelmann category (cat) and topological complexity (TC). We recall the discrete analogues of these invariants and demonstrate that $\text{cat}_r$ and $\text{TC}_r$ are special cases of $D_r$. By utilizing these three theorems, many significant results regarding the $\text{cat}_r$ and $\text{TC}_r$ can be derived.

In Section 4, we prove three propositions to establish useful upper and lower bounds for $D_r$. These propositions illustrate the behaviour of $D_r$ for certain compositions of maps.  
In Section 5, we demonstrate that the discrete homotopic distance is homotopy invariant up to a coefficient in certain cases. More precisely, the value of $D_r(f, g)$ for two maps $f$ and $g$ changes by a factor of $r$ if the maps are composed with a map having a homotopic inverse.


\section{Discrete homotopic distance}

In this section, we intend to introduce discrete homotopic distance for two Lipschitz maps. To define this distance, we need first to recall the notion of $(s, r)$-homotopy, a generalization of $r$-homotopy, for two scales $s,r > o$. 
The notion of $r$-homotopy, $r >0$, was defined by Barcelo as a relation between two $r$-Lipschitz maps, namely  $f, g: X \to Y$. The map $f$ is called $r$-homotopic to $g$ if there exists a non-negative integer $m$ and map $F: X \times [m] \to Y$ such that $F(-,0) = f$ and $F(-,m) = g$  whenever $F(-,i)$, for $i \in [m]$, and $F(x,-)$, for $x \in X$, are $r$-Lipschitz maps  \cite{Barcelo}. 
Note that the set $[m] := \{0, \ldots, m\}$ is equipped with the metric $d(a, b)=|a - b|$ and the Cartesian product $X \times [m] $ is equipped with the $\ell^1$-metric.
It is known that for metric spaces $ X, Y$  and $r > 0$, a map $f: X \to Y$ is called $r$-Lipschitz if $d_Y(f(x_1), f(x_2)) \leq r d_X(x_1, x_2)$, for every $x_1, x_2 \in X$. 
The most important difference between homotopy and $r$-homotopy is the values admitted by the parameter of time; For the classic homotopy,  time varies continuously between $t =0$ and $t=1$. But for the $r$-homotopy, time takes certain moments $t=0, t= 1, \ldots, t = m$. The notion of $r$-homotopy was generalized to $(s,r)$-homotopy as follows.

\begin{definition}[\cite{hasanzadeh}]\label{de2.1}
Let $s,r >0$ be real numbers. A map $F: X\times[m]\rightarrow Y$ is an $(s, r)$-homotopy between $f$ and $g$, denoted by $F: f \simeq_{(s,r)} g$, if
\begin{itemize}
\item [1.]
For all $x$ in $X$, $F(x, -):[m] \to Y$ is an $r$-Lipschitz map,
\item [2.]
For every $i\in [m]$, the map $F(-, i): X \to Y$ is an $s$-Lipschitz map,
\item [3.]
$F(x, 0) = f(x)$,
 and
$F(x, m) = g(x)$.
\end{itemize}

\end{definition}

As mentioned in the definition of $r$-homotopy,  the step length  is  the constant $r$ for vertical maps, $F(x, -)$'s, and horizontal maps $F(-, i)$'s. But in Definition \ref{de2.1}, the step length of   $F(x, -)$'s is equal to $r$ and for $F(-, i)$'s, it equals $s$. By this modification, one can extend  and use discrete homotopy theory for classifying spaces and maps  more consistently. Thus we need to know whether the $(s,r)$-homotopy is an equivalence relation.
Clearly the relation of $(s,r)$-homotopy is  reflexive and symmetric. Proposition \ref{pr2.2} verifies the transitivity of this relation.

\begin{proposition}\label{pr2.2}
Let $s, r >0$ be two real numbers and $f,g: X \to Y$ be two maps.
If $f\simeq_{(s,r)} g$ and $g\simeq_{(s,r)} h$, then $f\simeq_{(s,r)} h$. 
\end{proposition}
\begin{proof}
Let $f\simeq_{(s,r)} g$ and $g \simeq_{(s,r)} h$. Then there exist  $(s,r)$-homotopy functions $F: X \times [n] \to Y$ and $G: X \times [m] \to Y$ such that $F(x, 0)=f$, $F(x, n)=g$, $G(x, 0) = g$ and $G(x, m) = h$. Define $H: X \times [n+m] \to X$ by   
\[ 
H(x,i):= 
\begin{cases} 
 F(x, i) & 0 \leq i <n\\ 
G(x, i-n) & n \leq i \leq n+m\\ 
\end{cases}.
\] 
Obviously $H(x, 0) = f$ and $H(x, n+m) = h$. it remains to show that $H(x, -)$, for $x \in X$, is $r$-Lipschitz and $H(-,i)$, for $i \in [n+m]$, is $s$-Lipschitz. First we show that $H(x, -)$ is an $r$-Lipschitz function for $x \in X$.
For the cases $i, i' <n$ and $i, i' > n$, the inequality $d_Y (H(x, i), H(x, i')) \leq rd(i, i')$ holds, because $F(x, -)$ and $G(x, -)$ are $r$-Lipschitz functions.
If $i<n, i' > n$, since $d_Y$ is a metric satisfying the triangle inequality and $F(x,n) = g = G(x, 0)$, then
\begin{align*}
d_{Y} (H(x, i), H(x, i')) & = d_Y (F(x,i), G(x, i'-n))\\
 & \leq d_{Y}(F(x,i), F(x,n))+ d_{Y}(F(x,n), G(x, i'-n) )\\
 &=  d_{Y}(F(x,i), F(x,n))+ d_{Y}(G(x, 0), G(x, i'-n))\\
 & \leq r d(i,n)+r d(0,i'-n) \\
 &=r(n-i)+ r( i'-n)\\
&  =r (i'-i)= r d(i, i').
  \end{align*}
Moreover, $H(-, i)$ is an $s$-Lipschitz function, because it is obtained by gluing two functions $F(-,i)$ and $G(-,i$-$n)$ being $s$-Lipschitz due to the fact that $F$ and $G$ are $(s,r)$-homotopies. Therefore $H$ is an $(s,r)$-homotopy between $f$ and $h$.
\end{proof} 

The homotopic distance between two maps counts the minimum number of rules needed to make them locally homotopic; That is the smallest integer $k$ for which an open cover $U_0, \dots, U_k$ of the space $X$ exists such that the restricted maps of $f$ and $g$ to each $U_i$ are homotopic; $f|_{U_i} \simeq g|_{U_i}$ for $i =0 ,\ldots, k$ \cite{Mac'ias}. Now we present discrete homotopic distance to measure how far two maps are from each other, in terms of discrete homotopy.

\begin{definition}\label{dD.1} 
Let $X$ and $Y$ be two metric spaces and $f, g: X \to Y$ be two $s$-Lipschitz maps. The discrete homotopic distance between $f$ and $g$, denoted by $D_r(f, g)$, is the smallest integer $k$ for which there exists an open cover $U_0, \dots, U_k$ of $X$ such that the restrictions $f|_{U_i}$ and $g|_{U_i}$ are $(s, r)$-homotopic; $f|_{U_i} \simeq_{(s,r)} g|_{U_i}$ for all $i=0, 1, \dots, k$. If such a covering does not exist, we put $D_r(f, g):=\infty$.
\end{definition}

Note that in Definition \ref{dD.1}, the time required to deform from $f|_{U_i}$ to $g|_{U_i}$ depends on $U_i$ and varies by $i$. But one can equalize the end time if the homotopy map is extended axiomatically by repeating  the last component. 
More precisely, let $D_r(f, g)=k$. Then for $0 \leq i \leq k$, there exists map $\bar{H}_i: U_i \times [m_i] \to Y$ such that $\bar{H}_i: f|_{U_i} \simeq_{(s,r)} g|_{U_i}$. Put $m=\max \{ m_i\}_{i=0} ^ {i=k}$. We can extend $\bar{H}_i$ to $H_i: U_i \times [m] \to Y$ by $H_i(b, n)=g(b)$ for $b \in U_i$ and $m_i \leq n \leq m$.
Some basic properties of $D_r$ hold by Definition \ref{de2.1} and \ref{dD.1} as follows.
\begin{remark} \label{pD.1}
Let $X$ and $Y$ be two metric spaces and $f, g, f', g' : X \longrightarrow Y$ be $s$-Lipschitz maps. 
\begin{itemize}
\item [1.]
$D_r(f, g) =0$ if and only if $f$ is $(s, r)$-homotopic to $g$.
\item [2.]
 $D_r(f, g)=D_r(g,f)$.
\item [3.]
If $f\simeq_{(s,r)} f'$ and $g\simeq_{(s,r)} g'$, then $D_r(f, g)=D_r(f', g')$. 
\end{itemize}
\end{remark}

In the next example, the rules of two maps are expressed, for which the homotopic distance $D$ and the discrete homotopic distance $D_r$ admit different values. In fact there are two maps $f$ and $g$ such that $D_r(f,g)=0$ for every positive real number $r$, but $D(f, g) > 0$. This proves that neither the quantity of $D$ is a generalization of $D_r$ and nor vice versa.

\begin{example}\label{ex 3.3} 
Let $n$ and $k$ be positive integers. Consider two maps $f, g: S^1 \to \mathbb{C}^{\ast}= \mathbb{C} \setminus {\overline{0}}$ defined by $f(z)= z^{n}$ and $g(z)= z^{k}$ respectively. 
First we show that $f$ is an $n$-Lipschitz map and by a similar argument $g$ is a $k$-Lipschitz map. For every $z_1, z_2 \in S^1$
\begin{align*} 
d  (f(z_1), f(z_2)) & =d_{S^1}( z_1^n, z_2^n) = |z_1^n - z_2^n| \\
& =\lvert (z_1 - z_2)(z_1^{n-1}+z_1^{n-2}z_2 + z_1^{n-3}z_2^{2}+ \dots, z_2^{n-1}) \rvert\\ 
& \leq |z_1 - z_2|(|z_1^{n-1}|+|z_1^{n-2}z_2|+|z_1^{n-3}z_2^{2}|+ \dots, |z_2^{n-1}| ) \\ 
& = |z_2-z_1|(|z_1|^{n-1}+|z_1||z_2|^{n-2}+|z_1|^{n-3}|z_2|^{2}+ \dots, |z_2|^{n-1} )\\
& =n|z_2-z_1|. 
\end{align*}
Let $r>0$ be an arbitrary real number. Put $m=2[\frac{2 }{r}]+1$. We define $F: S^1 \times [m] \to \mathbb{C}^{\ast}$ by $F(z,j)=\frac{m-j}{m} z^n+ \frac{j}{m} z^k$ for $j \in [m]$ and $z \in S^1$. Note that $F(z,0)=f(z)= z^{n}$ and $F(z,m)=g(z)= z^{k} $. First, we investigate that $F$ maps $S^1 \times [m]$ into the space $\mathbb{C}^*$. If $F(z, j)=0$, we have $\frac{m-j}{m} z^n+ \frac{j}{m} z^k=0$. Then $|\frac{m-j}{j}||z^{n-k}| =1$. Therefore $m=2j$, an even number while $m$ is an odd number and this is a contradiction.
The following sequence of inequalities shows that $F(-, j)$, $j \in [m]$, is an $s$-Lipschitz function for $s= \max \{k, n \}$. 
 For every  $z_1, z_2 \in S^1$ 
\begin{align*} 
d  (F(z_1, j), F(z_2, j)) 
& =d_{S^1} (\frac{m-j}{m} z_2^n+ \frac{j}{m} z_2^k, \frac{m-j}{m} z_1^n+ \frac{j}{m} z_1^k) \\ 
&=| \frac{m-j}{m} (z_2^n-z_1^n)+ \frac{j}{m} (z_2^k -z_1^k)| \\ 
&= \bigg| \frac{m-j}{m} (z_2-z_1)(z_2^{n-1}+z_2^{n-2} z_1+ \dots + z_1^{n-1}) \\
& \quad  +\frac{j}{m} (z_2-z_1)(z_2^{k-1}+z_2^{k-2} z_1  + \dots + z_1^{k-1})\bigg|\\
&=\bigg|\frac{m-j}{m}(z_2^{n-1}+z_2^{n-2} z_1+ \dots + z_1^{n-1})
\\
& \quad +\frac{j}{m}(z_2^{k-1}+z_2^{k-2} z_1 
+ \dots + z_1^{k-1}) \bigg| | z_2-z_1| \\
&\leq  \frac{m-j}{m} \big(|z_2|^{n-1}+|z_2|^{n-2}| z_1|+ \dots + |z_1|^{n-1}\big) |z_2 -z_1| \\ 
& \quad +\frac{j}{m} \big(|z_2|^{k-1}+|z_2|^{k-2}| z_1|+ \dots + |z_1|^{k-1}\big)| z_2-z_1|\\
& \leq | \frac{m-j}{m}(n )+\frac{j}{m}(k) | | z_2-z_1| \\
&\leq \max \lbrace n ,k \rbrace |z_2-z_1|\\
&=sd(z_2 , z_1). 
\end{align*} 

Here we show that $F(z,-)$ is an $r$-Lipschitz function, making $f$ and $g$ $(s,r)$-homotopic. For $0 \le j_1, j_2 \le m$
\begin{align*} 
d  (F(z, j_2), F(z, j_1)) 
&=d (\frac{m-j_2}{m} z^n+ \frac{j_2}{m}z^k , \frac{m-j_1}{m} z^n+ \frac{j_1}{m}z^k) \\ 
&=\left| - \frac{j_2-j_1}{m}z^n+ \frac{j_2-j_1}{m}z^k \right|\\
&= \frac{|j_2-j_1| }{m}|z^k-z^n| \\ 
& \leq \frac{|j_2-j_1|}{m} (|z^k|+|z^n|) \\
& = \frac{2}{m} |j_2 - j_1|
\\
& \leq r d(j_2 , j_1). 
\end{align*} 
Therefore $D_r(f, g)=0$, but we show that $D(f, g) > 0$. If $D(f, g)=0$, then there is a map $G: S^1 \times I \to \mathbb{C}^{\ast}$ such that $G: f \simeq g$. Define $f_1, g_1: S^1 \to S^1$ by the rule $f_1(z)=\frac{f(z)}{\parallel f(z) \parallel}$ and $g_1(z)=\frac{g(z)}{\parallel g(z) \parallel}$ and also define $H: S^1 \times I \to S^1$ by $H(z,t)=\frac{G(z,t)}{\parallel G(z,t) \parallel}$. Hence $D(f_1, g_1) = 0$. But this is a contradiction, because  $D(f_1, g_1) \neq 0$ by \cite[Theorem 2.1]{Vergilia}.
\end{example}

In Figure \ref{f1}, there are two holes in the path from f to g, and then $D(f,g)=2$. But the discrete homotopic distance changes with the value of $r$. Let the diameter of holes $h_0$ and $h_1$ equal $d_0$ and $d_1$ respectively and $d_0 < d_1$. Then the discrete homotopic distance $D_r(f, g)$ varies between $0$ and $2$ depending on the step length $r$ and satisfies the following
\[
D_r(f,g) = \begin{cases}
0\ \ \mathrm{if}\  & r > d_1 \\
1\  \ \mathrm{if}\   & d_0 <r \le d_1 \\
2\ \ \mathrm{if}\  & r \le d_0
\end{cases}.
\]

\begin{figure}[!ht]
\centering
\includegraphics[scale=0.5]{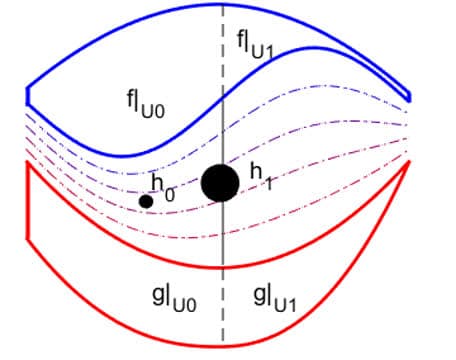} 
\caption{$D_r(f,g)$ in a space including holes.}\label{f1}
\end{figure}

In other words, by  longer steps, more holes can be passed as well as shown in Proposition \ref{pD.0}. In fact $D_r$ inversely depends on $r$, or equivalently the discrete homotopic distance decreases whenever the scale $r$ increases.

\begin{proposition} \label{pD.0}
Let $s, r_1, r_2$ be positive real numbers and  $f, g: X \to Y$ be two $s$-Lipschitz maps. If $r_1 \leq r_2$, then $D_{r_2}(f, g) \leq D_{r_1}(f, g)$. 
\end{proposition}

\begin{proof} 
Let $D_{r_1}(f, g)=k$. Then there exists an open cover $U_0, \ldots, U_k$ and for each $i=0, \ldots,k$, an  $(s, r_1)$-homotopy $F_i: U_i \times [m] \to Y$ such that $F_i: f|_{U_i} \simeq_{(s,r_1)} g|_{U_i}$. 
Since $F_i$ is an $(s,r_1)$-homotopy, $F_i(x, -)$ is $r_1$-Lipschitz. Thus $d_{Y} (F_i(x, j_2), F_i(x, j_1)) \leq r_1 d(j_2, j_1)$ for $0 \leq j_1, j_2 \leq m$.
Moreover $r_1 \le r_2$, and hence 
$d_{Y} (F_i(x, j_2), F_i(x, j_1)) \leq r_2 d(j_2, j_1)$ for $0 \leq j_1, j_2 \leq m$.
Therefore  $F_i(x, -)$ is $r_2$-Lipschitz, and then
 $F_i: f|_{U_i} \simeq_{(s,r_2)} g|_{U_i}$. This implies that $D_{r_2}(f, g) \leq k=D_{r_1}(f, g)$.
\end{proof}

\section{Discrete categories and relation to discrete distance}

In this section, we discuss the relation between discrete homotopic distance and some categories of map. It is well-known that the Lusternik-Schnirelmann category  $cat$ counts how far a space is from contractiblity. In the literature of homotopic distance, $cat (X) = D(id_X, c)$. Also for topological complexity, $TC(X) = D(p_1,p_2)$ where $p_1, p_2: X \times X \to X$ are projection maps. Similar statements corresponded to discrete versions   hold which are surveyed below.  
First we recall  the concept of $r$-null-homotopy, providing a discrete interpretation of null-homotopy. 
\begin{definition}[\cite{hasanzadeh}]\label{dD.2}
Let $f: X \to Y$ be an $s$-Lipschitz map. The map $f$ is called $r$-null-homotopic if there exists an $(s, r)$-homotopy map $F: X\times[m] \to Y$ between $f: X \to Y$ and some constant map. Specifically, the identity map on $X$ is termed $r$-null-homotopic, if there is a $(1, r)$-homotopy map between the identity map and some constant map. In this case, the space $X$ is called $r$-contractible.
\end{definition}

To define Lusternik-Schnirelmann category in our point of view, we rewrite the initial definition with discrete tools.
If $X$ is path-connected, then the Lusternik-Schnirelmann category of $X$, denoted by $cat(X)$, is defined as the smallest non-negative integer $k$ for which there exists an open covering $U_0, \dots, U_k$ for the space $X$ such that each ${U_i}$ is contractible in $X$, for $i = 0, 1, \dots, k$. If such a cover does not exist, then $cat(X):=\infty $. Also the Lusternik-Schnirelmann category of map $f: X \to Y$, denoted by $cat(f)$, is the smallest non-negative integer $k$ such that there exists an open cover $U_0, \dots, U_k$ for the space $X$ such that $f|_{U_i}$ is homotopic to some constant map. If such a cover does not exist, then $cat(f):=\infty $. In particular, $cat(id_X)=cat(X)$ if $id_X$ denotes the identity map \cite{Pavesiˇc}.
Additionally, we recall that the discrete Lusternik-Schnirelmann category of $X$, denoted by $cat_r(X)$, is the smallest non-negative integer $k$ for which there exists an open cover $U_0, \dots, U_k$ of $X$ such that each ${U_i}$ is $r$-contractible in $X$ for $i = 0, 1, \dots, k$. If such a cover does not exist, then $cat_r(X):=\infty$ \cite{hasanzadeh}. 
A subset $U$ is called $r$-contrcatible in $X$ if there exists $(1,r)$-homotopy $H:U \times [m] \to X$ such that $H: id_U \simeq_{(1,r)} c_x$, where $c_x$ denotes the constant map on $x$ for some $x \in U$.

\begin{definition}\label{dD.3}
The discrete Lusternik-Schnirelmann category of map $f: X \to Y$, denoted by $cat_r(f)$, is the smallest non-negative integer $k$ such that there exists an open cover $U_0, \dots, U_k$ for the space $X$ such that $f|_{U_i}$ is $(s,r)$-homotopic to some constant map. If such a cover does not exist, then $cat_r(f):=\infty $. In particular if $id_X$ denotes the identity map, then $cat_r(id_X)=cat_r(X)$.
\end{definition} 
Let $X$ be a metric space. A map $\gamma: [m] \to X$ is called an $r$-path if $ d_X(\gamma (i), \gamma (i + 1)) \leq r$ for every $i \in [m-1]$. If for each $x, y \in X$, there exists an $r$-path from $x$ to $y$, then X is said to be $r$-connected \cite{Barcelo}. We know that each constant map is $s$-Lipschitz, for each $s \ge 0$. Thus by any $r$-path between two points $x$ and $y$, one can define an $(s,r)$-homotopy between constant maps $c_x$ and $c_y$. This fact 
is important to prove Proposition \ref{pD.2}, which  presents some natural relations between $cat_r$ and $D_r$.
In fact $cat_r$ is the distance between the given map and constant map.
 
\begin{proposition}\label{pD.2}
Let $X$ be an $r$-connected space and $c_x$ denote the constant map on $x \in X$. \begin{itemize}
\item [1.] 
Define $i_1, i_2: X \to X \times X$ by $i_1(x)=(x,x_0)$ and $i_2(x)=(x_0, x)$ for some fixed point $x_0 \in X$. Then
\begin{equation}\label{eq3.3}
cat_r(X)=D_r(id_X, c_{x_0}) = 
D_r(i_1,c_{(x_0,x_0)})= D_r(i_2,c_{(x_0,x_0)})
=D_r(i_1, i_2).
\end{equation}
\item [2.]
 Let $f: X \to Y$ be an $s$-Lipschitz map. Then
$cat_r(f)= D_r(f, c_x)$.
\end{itemize}
\end{proposition}
 
\begin{proof}
\begin{itemize} 
\item [1.]  
To prove equalities of \eqref{eq3.3}, first we show that each term is less than or equal to the next one
\[
cat_r(X) \le D_r(id_X, c_{x_0}) \le
D_r(i_1,c_{(x_0,x_0)}) \le  D_r(i_2,c_{(x_0,x_0)})
\le D_r(i_1, i_2).
\]
Then we verify that
$D_r(i_1, i_2) \le cat_r(X)$, 
which makes the  sequence of equalities \eqref{eq3.3} to be maintained.
Let $D_r(\text{id}_X, c_{x_0}) = k$. Then there exists an open cover $\{U_0, \ldots, U_k\}$ of $X$ such that the restriction of the identity map is $(1,r)$-homotopic to $c_{x_0}$; That is $id|_{U_i} \simeq_{(1,r)} c_{x_0}$. Therefore each $U_i$ is $r$-contractible in $X$.   Since $cat_r(X)$ 
 is the minimum cardinality minus $1$, of open covers of $X$ with $r$-contractible elements, $cat_r(X) \le k$.

Now, assume that $D_r(i_1,c_{(x_0,x_0)}) = k$. This means there exists an open cover $\{U_0, \dots, U_k\}$ of $X$ such that $i_1|_{U_i} \simeq_{(1, r)} c_{(x_0, x_0)}|_{U_i}$ for each $i = 0, 1, \dots, k$. Let $p_1: X \times X \to X$ be the projection map onto the first coordinate. For each $x \in U_i$, we have 
$p_1 \circ i_1(x) = p_1(x, x_0) = x = id_X(x)$, and $p_1 \circ c_{(x_0,x_0)} =c_{x_0}$. Since projection  map $p_1$ is $1$-Lipschitz, by  Lemma \ref{pC.0},
\[
id_X = p_1 \circ i_1 \simeq_{(1,r)} p_1 \circ c_{(x_0,x_0)} = c_{x_0}.
\]
Therefore $D_r(id_X, c_{x_0}) \le
D_r(i_1,c_{(x_0,x_0)})$.

Let $D_r(i_2, c_{(x_0,x_0)}) = k$. 
This means there exists an open cover $\{U_0, \dots, U_k\}$ of $X$ such that $i_2|_{U_i} \simeq_{(1, r)} c_{(x_0,x_0)}|_{U_i}$ for each $i = 0, 1, \dots, k$. Let $p_2: X \times X \to X$ be the projection map onto the second coordinate. For each $x \in U_i$, we have 
\[i_1 \circ p_2 \circ i_2(x) = i_1 \circ p_2(x_0, x) = i_1(x),\]
 and $i_1 \circ p_2 \circ c_{(x_0,x_0)} =i_1 \circ c_{x_0} = c_{(x_0,x_0)}$. Since $i_1$ and the projection  map $p_1$ are $1$-Lipschitz, by  Lemma \ref{pC.0},
\[
i_1|_{U_i} = i_1 \circ p_2 \circ i_2|_{U_i} \simeq_{(1,r)} i_1 \circ p_2 \circ c_{(x_0,x_0)}|_{U_i} = c_{(x_0, x_0)}|_{U_i}.
\]
Hence $D_r(i_1, c_{(x_0, x_0)}) \le D_r(i_2, c_{(x_0, x_0)})$.

Now, assume that $D_r(i_1, i_2) = k$. This means there exists an open cover $\{U_0, \dots, U_k\}$ of $X$ such that $i_1|_{U_i} \simeq_{(1, r)} i_2|_{U_i}$ for each $i = 0, 1, \dots, k$. If  we compose both sides of the homotopy relation with $i_2 \circ p_2$, for $ =0, \ldots,k$, we have 
\begin{align*}
i_2 \circ p_2 \circ i_1 (x) &= i_2 \circ p_2 (x, x_0) = i_2(x_0) = (x_0,x_0) = c_{(x_0,x_0)} (x), \textrm{ and }\\
i_2 \circ p_2 \circ i_2 (x) &= i_2 \circ p_2 (x_0, x) = i_2(x).
\end{align*}
Therefore
\[
c_{(x_0,x_0)}|_{U_i} = i_2 \circ p_2 \circ i_1|_{U_i} \simeq_{(1,r)} i_2 \circ p_2 \circ i_2|_{U_i} = i_2|_{U_i}.
\]
Hence $D_r( i_2, c_{(x_0,x_0)}) \le k$.

Finally assume that $cat_r(X)=k$. 
Then there exists a cover $\{U_0, \ldots, U_k\}$ of $X$ consisting of $r$-contractible open sets; That is $id|_{U_i} \simeq_{(1,r)} c_{x}$ for some $x \in U_i$. Since $X$ is $r$-connected, $x$ can be any point in $X$, namely $x_0$. Then since $i_1$ and $i_2$ are $1$-Lipschitz, by Lemma \ref{pC.0}, 
\begin{align*}
i_1|_{U_i} & = i_1 \circ id|_{U_i} \simeq_{(1,r)} i_1 \circ c_{x_0} = c_{(x_0,x_0)}
\\
i_2|_{U_i} &= i_2 \circ id|_{U_i} \simeq_{(1,r)}     i_2 \circ c_{x_0} = c_{(x_0, x_0)}.
\end{align*}
Proposition \ref{pr2.2} implies that $i_1|_{U_i} \simeq_{(1,r)} i_2|_{U_i}$. Since $D_r(i_1, i_2)$ is the minimum cardinality minus $1$ of such open covers of $X$, we conclude $D_r(i_1, i_2) \leq \text{cat}_r(X)$.

\item
The distance $D_r(f, c)$ is the minimum cardinality  minus $1$ of open covers of $X$ such that $f$ is $(s, r)$-homotopic to $c$ on each element of cover. Moreover,  $\text{cat}_r(f)$ is the minimum cardinality  minus $1$ of covers $\{U_0, \ldots,U_k\}$ of $X$ such that $f|_{U_i}$ is $(s,r)$-homotopic to some constant map in $U_i$ for $i =0, \ldots,k$. We show that since $X$ is $r$-connected, these two conditions are equivalent.
 In an open set $U$, map  $f$ is $r$-contractible if and only if $f|_U \simeq_{(s,r)} c_x$ for some $x \in U$ if and only if  $f|_{U}$ is $(s, r)$-homotopic to $c$, because $X$ is $r$-connected and all constant maps are $(s,r)$-homotopic for each real positive number $s$. Thus, we have $D_r(f, c) = \text{cat}_r(f)$. 
\end{itemize} 
\end{proof}

We recall that the topological complexity of path connected space $X$ is the smallest number $k$ such that $X \times X = U_0 \cup U_1 \cup \dots \cup U_k$ in which for each $i \in \{ 0, 1, \dots, k \}$ there exists a motion planning $s_i: U_i \to X^I$ such that $\pi \circ s_i = \text{id}_{U_i}$. If no such $k$ exists, then $TC(X)=\infty$ \cite{Farber}. 
The discrete topological complexity of discrete motion is defined and studied in \cite{hasanzadeh}. 
Let $X$ be a metric space and $r$ be a positive real number. Then the discrete topological complexity of $X$, denoted by $TC_r(X)$, is the smallest number $k$ such that $X \times X = U_0 \cup U_1 \cup \dots \cup U_k$ and  for each $i \in \{ 0, 1, \dots, k \}$, there is $1$-Lipschitz map $S_i: U_i \to X^{[m]}$ such that $\pi \circ S_i = \text{id}_{U_i}$. If no such $k$ exists, then $TC_r(X)=\infty$ \cite{hasanzadeh}. Note that the space $X^{[m]}$ denotes the set
of $r$-paths $\gamma : [m] \to X$ equipped with the uniform metric. Proposition \ref{cD.1} represents the discrete topological complexity as the distance between two projection maps. A similar statement was previously proven in \cite{Mac'ias}:
If $X$ is a topological space, and $p_1, p_2: X \times X \to X$ are projection maps defined by $p_1(x_1,x_2)=x_1$ and $p_2(x_1,x_2)=x_2$, then $TC(X)=D(p_1, p_2)$.

\begin{proposition}\label{cD.1}
Let  $X$ be a metric space and $p_1, p_2: X \times X \to X$ be projection maps. Then $TC_r(X)=D_r(p_1, p_2)$. 
\end{proposition} 

\begin{proof} 
Assume that $TC_r(X) = k$. Then there exists an open cover $\{U_0, U_1, \dots, U_k\}$ of $X \times X$ such that for each $U_i$, there is a   $1$-Lipschitz map $S_i: U_i \to X^{[m]}$ satisfying $\pi \circ S_i = \text{id}_{U_i}$. For $i=0, \ldots,k$, define $H_i: U_i \times [m] \to X$ by $H_i(x, y, j) = S_i(x, y)(j)$ for  any $x, y \in U_i$  and $j \in [m]$. 
The map $H_i$, $i = 0, \ldots,k$, satisfies the boundary conditions
\begin{align*}
H_i(x, y, 0) &= S_i(x,y)(0) = x = p_1(x, y)\\
H_i(x, y, m) &= S_i(x,y) (m) = y = p_2(x, y),
\end{align*}
because the equality $\pi \circ S_i = id|_{U_i}$ implies that $S_i(x,y)$ is an $r$-path from $x$ to $y$ for $(x,y) \in U_i$. Also $S_i$ is $1$-Lipschitz and so is $H(-, j)$, $j =0,
\ldots, m$, due to the uniform metric on $X^{[m]}$
\begin{align*}
d_X (H_i (x_1, y_1, j), H_i (x_2, y_2, j)) &= d_X( S_i(x_1, y_1) (j), S_i(x_2, y_2) (j))\\
& \le  \sup_{j' \in [m]} d_X (S_i(x_1, y_1)(j'), S_i (x_2, y_2)(j')) \\
& = d_{X^{[m]}} (S_i(x_1, y_1), S_i (x_2, y_2))\\
& \le d_{X \times X} \big((x_1, y_1), (x_2, y_2) \big).
\end{align*}
Note that $p_1$ and $p_2$ are $1$-Lipschitz maps, and according to this fact we find a $(1,r)$-homotopy connecting them.
Moreover, $H((x,y),-)$ is $r$-Lipschitz, because $S_i (x,y)$ is an $r$-path. Hence, the open cover $\{U_0, U_1, \dots, U_k\}$ admits a $(1,r)$-homotopy $H_i$ on each $U_i$. Since $D_r(p_1, p_2)$ is defined as the minimal number of such open sets admitting a $(1,r)$-homotopy between $p_1$ and $p_2$, we conclude that $D_r(p_1, p_2) \leq TC_r(X)$.
 
Now, assume  that $D_r(p_1, p_2) = k$. Then there exists an open cover $\{V_0, V_1, \dots, V_k\}$ of $X \times X$ such that on each $V_i$, there exists a $(1,r)$-homotopy $H_i: V_i \times [m] \to X$ satisfying $H_i(x,y, 0) = p_1(x,y) =x$ and $H_i(x,y, m) = p_2(x,y) = y$ for all $(x,y) \in V_i$. 
Define motion planning $S_i: V_i \to X^{[m]}$ by setting $S_i(x, y) (j) = H(x,y,j)$ for $j \in [m]$,  the $r$-path in $X$ traced by the $(1,r)$-homotopy $H_i$ between the points $p_1(x)$ and $p_2(x)$. Since $S_i (x,y) (0) = H_i(x, y, 0) = p_1(x,y) = x$ and $S_i(x,y) (m) = H_i(x, y, m) = p_2(x,y) = y$,
\[
\pi \circ S_i(x,y) = (S_i(x,y) (0), S_i(x,y) (m)) = (x,y) = id_{X \times X} (x,y), 
\]
and hence this defines a valid section of the projection $\pi$. 
 Thus, the open cover $\{V_0, V_1, \dots, V_k\}$  of $X \times X$ admits sections $S_i: V_i \to X^{[m]}$, implying that $TC_r(X) \leq D_r(p_1, p_2)$. 
 Since we show that $TC_r(X) \leq D_r(p_1, p_2)$ and $D_r(p_1, p_2) \leq TC_r(X)$, it follows that $TC_r(X) = D_r(p_1, p_2)$. 
\end{proof}

It is noted that similar symbols, $TC_r$ and $TC_{r,s}$ were used in \cite{Rudiak, Zapata}, but they differ  in their definitions, values, uses and behaviours.

\section{Discrete homotopic distance of compositions} 
In this section, we study the behaviour of $D_r$ for certain compositions of maps. 
Let $X, Y, Z$ be topological spaces. Consider $f, g: X \to Y$ , $k: Y \to Z$ and $h, h': Z \to X$ be maps. It is proven that $D(f \circ h, g \circ h) \leq D(f, g)$, $D(f \circ h, g \circ h') \leq D(h, h')$ and $D(h \circ f, h \circ g) \leq D(f, g)$ \cite{Mac'ias}. We establish these relations for the discrete case. Using the following theorems, many important results about the $cat_r$ and $TC_r$ of maps can be derived. 
Specially by these inequalities, we present some valuable upper bounds for $cat_r$ and $D_r$.

\begin{lemma}\label{pC.0} 
Let $f, g: X \to Y$ be $s_1$-Lipschitz maps, and let $h, h': Y \to Z$ be $s_2$-Lipschitz maps for $s_1, s_2, r \ge 0$. Then:
\begin{itemize} 
\item [1.] 
If $f$ and $g$ are $(s_1, r)$-homotopic, then $h \circ f \simeq_{(s_2s_1, s_2r)} h \circ g$. In particular, if $s_2 \leq 1$, then $h \circ f$ and $h \circ g$ are $(s_1, r)$-homotopic. 
\item [2.] 
If $h$ and $h'$ are $(s_2, r)$-homotopic, then $h^\prime \circ f$ and $h \circ f$ are also $(s_1s_2, r)$-homotopic. 
\end{itemize} 
\end{lemma} 
\begin{proof} 
\begin{itemize} 
\item [1.] 
Consider a map $F: X \times [m] \to Y$ as an $(s_1, r)$-homotopy between $f$ and $g$. Define a map $K: X \times [m] \to Z$ by $K(x, i) = h(F(x, i))$. Thus, $K(x, 0) = h(f(x))$ and $K(x, m) = h(g(x))$. 
We show that $K(-, i)$ is an $s_2 s_1$-Lipschitz map. For every $ x_1, x_2 \in X$, and $i \in [m]$
\[ 
\begin{aligned} 
 d(K(x_2, i), K(x_1, i)) & = d(h(F(x_2, i)), h(F(x_1, i))) \\ 
& \leq s_2 d(F(x_2, i), F(x_1, i)) \leq s_2 s_1 d(x_2, x_1). 
\end{aligned} 
\] 
Moreover, $K(x,-)$ is an $r$-Lipschitz map, because for every $ i_1, i_2 \in [m]$ and $x \in X$, we have the following and hence $h \circ f$ and $h \circ g$ are $(s_1s_2, s_2r)$-homotopic;
\[ 
\begin{aligned} 
 d(K(x, i_2), K(x, i_1)) & = d(h(F(x, i_2)), h(F(x, i_1))) \\ 
& \leq s_2 d(F(x, i_2), F(x, i_1)) \leq s_2 r d(i_2, i_1).
\end{aligned} 
\] 
\item [2.] 
Consider a map $G: Y \times [m] \to Z$ as an $(s_2, r)$-homotopy between $h$ and $h^\prime$. Define a map $H: X \times [m] \to Z$ by  $H(x, i) = G(f(x), i)$. The map $H$ is well-defined, $H(x, 0) = h(f(x))$ and $H(x, m) =h^\prime(f(x))$. 
Now, we show that $H(-,i)$ is an $s_2 s_1$-Lipschitz map for each $i \in [m]$. For every $ x_1, x_2 \in X$,
\[ 
\begin{aligned} 
d(H(x_2, i), H(x_1, i)) & = d(G(f(x_2), i), G(f(x_1), i)) \\ 
& \leq s_2 d(f(x_2), f(x_1))\leq s_2 s_1 d(x_2, x_1).  
\end{aligned} 
\] 
Also, $H(x,-)$ is an $r$-Lipschitz map, because  for every $ i_1, i_2 \in [m]$,
\[d(H(x, i_2), H(x, i_1)) = d(G(f(x), i_2), G(f(x), i_1))\leq r d(i_2, i_1).\]
Therefore $h \circ f$ and $h^\prime \circ f$ are $(s_2s_1, r)$-homotopic. 
\end{itemize} 
\end{proof}

Lemma \ref{pC.0} states that the composition of maps preserves $(s,r)$-homotopy up to some scales multiplied. Now we use these properties to obtain information about the discrete homotopic distance between maps.

\begin{proposition}\label{pC.1} 
Let $X$, $Y$ and $Z$ be metric spaces, $f, g: X \to Y$ be $s_1$-Lipschitz maps and $h: Y \to Z$ be an $s_2$-Lipschitz map. Then $D_{s_2r}(h \circ f, h \circ g) \leq D_r(f, g)$. 
\end{proposition} 
\begin{proof} 
Assume that $D_r(f, g) = k$. There exists an open cover $\{U_0, \dots, U_k\}$ of $X$, such that $f|_{U_i} \simeq_{(s_1,r)} g|_{U_i}$ for each $i$. By Lemma \ref{pC.0}(1), $h \circ g|_{U_i} \simeq_{(s_1s_2, s_2r)} h \circ f|_{U_i}$. Since $(h \circ f)|_{U_i} = h \circ f|_{U_i}$ and $(h \circ g)|_{U_i} = h \circ g|_{U_i}$, it follows that $(h \circ f)|_{U_i} \simeq_{(s_1s_2, s_2r)} (h \circ g)|_{U_i}$. Therefore, $D_{s_2r}(h \circ f, h \circ g) \leq k = D_r(f, g)$.
\end{proof} 

It is well-known that the Lusternik-Schnirelmann category of a map is less than or equal to the Lusternik-Schnirelmann of the domain. This inequality in discrete case, is obtained  by using the discrete homotopy distance.

\begin{corollary}\label{cC.1} 
Let $X$ be an $r$-connected metric space and $f$ be an $s$-Lipschitz map with $s \le 1$. Then, $cat_r(f) \leq cat_r(X).$ 
\end{corollary}

\begin{proof} 
Proposition \ref{pD.2} items (1)  and (2) state that 
$cat_r(X) =D_r (id_X, c_{x_0})$ 
and  $cat_r(f) =D_r (f \circ id_X, f \circ c_{x_0})$ 
respectively. Substituting $h = f$,  $ f = id_X$ and 
$g= c_{x_0}$
into the inequality of Proposition
 \ref{pC.1}, gives 
$D_{sr}(f, c_{x_0})\leq D_r(id_X,c_{x_0})$. 
Since $s \le 1$, we have 
$sr \le r$, and then by Proposition \ref{pD.0}, 
$D_{r}(f, c_{x_0})\leq D_{sr} (f, c_{x_0})$. 
Combining these inequalities leads to the next sequence of inequalities and thus we are done; 
\[cat_r(f) = D_{r}(f, c_{x_0})\leq D_{sr} (f, c_{x_0}) \leq D_r(id_X,c_{x_0}) = cat_r(X).
\]
\end{proof}
 
Corollary \ref{cC.1} investigates discrete category of $s$-Lipschitz maps for $s \le 1$. If $s >1$, we have no such inequality.
Because the constant map is $s$-Lipschitz for any non-negative real number $s$ and $cat_r(c_x) =0 \le cat_r(X)$. Also there are $s$-Lipschitz maps being far from constant map; That is $cat_r(f)$ is large enough to be greater than $cat_r(X)$.

\begin{proposition}\label{pC.2} 
Let $X$, $Y$, and $Z$ be metric spaces, $f, g: X \to Y$ be $s_1$-Lipschitz maps, and $h: Z \to X$ be $s_2$-Lipschitz. Then  $D_r(f \circ h, g \circ h) \leq D_r(f, g)$.
\end{proposition} 
\begin{proof} 
Assume that $D_r(f, g)=k$. There exists an open cover $\{U_0, \dots, U_k\}$ of $X$ such that $f|_{U_i} \simeq_{(s_1,r)} g|_{U_i}$ for each $i = 0, 1, \dots, k$. Put $f_i=f|_{U_i}$ and $g_i=g|_{U_i}$. Then $f_i \simeq_{(s_1,r)} g_i$. 
 Consider the subset $V_i=h^{-1}(U_i) \subset Z$. The restricted map $h_i: V_i \to X$ can be expressed as a composition of the map $\overline{h}_i: V_i \to U_i$ defined by $\overline{h}_i(x)= h(x)$ and the inclusion map $i_{U_i} : U_i \subset X$.  Therefore, we have$(f \circ h)_i =f \circ h_i= f_i \circ \overline{h}_i$ and similarly $(g \circ h)_i =g\circ h_i= g_i \circ \overline{h}_i.$
By Lemma \ref{pC.0}(2), $f_i \circ \overline{h}_i \simeq_{(s_2s_1, r)} g_i \circ \overline{h}_i$. Therefore, $(f \circ h)_i=f \circ h_i= f_i \circ \overline{h}_i \simeq_{(s_2s_1,r)} g_i \circ \overline{h}_i = (g \circ i_{U_i} ) \circ \overline{h}_i= g \circ h_i = (g \circ h)_i$. Thus, $(f \circ h)_i \simeq_{(s_2s_1,r)} ( g\circ h)_i$. Since $D_r(f \circ h , g \circ h)$ is the minimum number of such sets, we conclude that $D_r(f \circ h, g \circ h)\leq D_r(f,g)$. 
\end{proof}

As previously mentioned, $TC_r$ equals the distance between projection maps. This fact implies that $TC_r$ admit discrete category as a lower bound. It was directly proven in \cite{hasanzadeh}.

\begin{corollary} \label{cC.2} 
Let $X$ be an $r$-connected space for $r \ge 0$. Then $cat_r(X) \leq TC_r(X)$. 
\end{corollary}

\begin{proof} 
Define projection maps $p_1, p_2: X \times X \to X$ by $p_1(x, x') = x$ and $p_2(x, x') = x'$, respectively. Let $x_0 \in X$ be a fixed point. Define $h:X \to X \times X$ by $h(x)=(x,x_0)$. Then we have $p_1 \circ h(x)=p_1(x,x_0)=x=id_X (x)$ and $p_2 \circ h(x) = p_2(x,x_0) = x_0 = c_{x_0} (x)$, for every $x \in X$. By Proposition \ref{pC.2}, we have $D_r(\text{id}_X, c_{x_0})=D_r(p_1 \circ h, p_2 \circ h) \leq D_r(p_1, p_2)$. By Corollary \ref{cD.1} $D_r(p_1, p_2) = TC_r(X)$. Also, by Proposition \ref{pD.2} (1), we know that $cat_r(X) = D_r(\text{id}_X, c_{x_0})$. Therefore, \[cat_r(X) = D_r(\text{id}_X, c_{x_0})= D_r(p_1 \circ h, p_2 \circ h) \leq D_r(p_1, p_2) = TC_r(X).\] 
\end{proof}

When the discrete category of a map is compared with the category of domain in Corollary \ref{cC.1}, some conditions on $s$ are needed. But for codomain, no condition on $s$ we require.

\begin{corollary} 
Let $X$ and $Y$ be metric spaces and $Y$ be an $r$-connected space. If $f: X \to Y$ is an $s$-Lipschitz map, then $cat_r(f) \leq cat_r(Y)$. 
\end{corollary}

\begin{proof} 
According to Proposition \ref{pD.2}(1) and (2), $cat_r(Y)=D_r(\text{id}_Y, c_{y_0})$ and $cat_r(f)=D_r(f, c_{y_0})$ respectively. By Proposition \ref{pC.2}, $D_r (\text{id}_Y \circ f, c_{y_0} \circ f) \leq D_r (\text{id}_Y, c_{y_0})$. Then \[cat_r(f) = D_r (\text{id}_Y \circ f, c_{y_0} \circ f) \leq D_r (\text{id}_Y, c_{y_0}) = cat_r(Y).\] 
\end{proof}

By Corollaries \ref{cC.1} and \ref{cC.2}, for $r$-connected spaces $X$ and $Y$ and $s$-Lipschitz map $f:X \to Y$, if $s \le 1$ then
$cat_r (X) \le cat_r (f) \le cat_r(Y)$.
The following proposition gives an upper bound for the discrete distance between 
two maps, by the amount  of their discrete categories; That is the distance between two maps is less than or equal to the distance travelled to transfer from the first one to the constant map and go back to the second one.

\begin{proposition}
If $f, g: X \to Y$ are $s$-Lipschitz maps and $Y$ is $r$-connected, then
\[D_r(f, g) \leq (cat_r(f) + 1) \cdot (cat_r(g) + 1) -1.\]
\end{proposition}

\begin{proof}
Assume that $cat_r(f) = k$ and $cat_r(g) = l$. Then by Proposition \ref{pD.2}, there exist open covers $\{U_0, \dots, U_k\}$ and $\{V_0, \dots, V_l\}$ of $X$ such that
for each $i \in \{0, \dots, k\}$, $f|_{U_i} \simeq_{(s, r)} c_y$, and
for each $j \in \{0, \dots, l\}$, $g|_{V_j} \simeq_{(s, r)} c_y$ for some $y \in Y$.
Consider the open cover $\{U_i \cap V_j \mid i \in \{0, \dots, k\}, j \in \{0, \dots, l\}\}$ of $X$. For each $U_i \cap V_j$,  $f$ and $g$ are $(s, r)$-homotopic to the constant map $c_y$ and they are homotopic to each other by Proposition \ref{pr2.2}.
Since there are $(k + 1) \times (l + 1)$ such intersections, we have
$D_r(f, g) \leq (cat_r(f) + 1) \cdot (cat_r(g) + 1) -1$,
because we start numbering open sets from zero. 
\end{proof}

Propositions \ref{pC.1} and \ref{pC.2} surveys the discrete distance between maps composed with some fixed map. Proposition \ref{pC.3} deals with the case where the maps composed are the same up to homotopy.

\begin{proposition}\label{pC.3}
Let $X$, $Y$, and $Z$ be metric spaces, $s_1, s_2, r \ge 0$, and $f, g: X \to Y$ be $(s_1, r)$-homotopic maps. Suppose that $h, h': Z \to X$ be any $s_2$-Lipschitz maps. Then
\[
D_{(s_1+1)r}(f \circ h, g \circ h') \leq D_r(h, h').
\]
\end{proposition}

\begin{proof}
Assume that $D_r(h, h') = k$. Then there exists an open cover $\{U_0, \dots, U_k\}$ of $Z$ such that $h|_{U_i} \simeq_{(s_2, r)} h'|_{U_i}$ for each $i = 0, \dots, k$. Specially, for $i=0, \ldots,k$, there exists $(s_2, r)$-homotopy $H_i:U_i \times [m_1] \to X$ such that 
\[
H_i(z,0) = h(z) \textrm{ and } H_i (z,m_1)= h'(z) \textrm{ for all } z \in U_i.
\]
Since $f$ and $g$ are $(s_1, r)$-homotopic, there exists $(s_1,r)$-homotopy $F: X \times [m] \to Y$ such that $H(x, 0) = f(x)$ and $H(x, m_2) = g(x)$.
Consider $m= \max\{m_1, m_2\}$ and the map $K: Z \times [m] \to Y$ defined by $K(z, i) = F(H(z, j), j)$.
Since $H(-,j)$ is $s_2$-Lipschitz and $F(-,j)$ is $s_1$-Lipschitz, the map $K(-,j)$ is $s_1s_2$-Lipschitz. 
Also, $K(z,-) = F(H(z,-),-)$ is $(s_1 +1)r$-Lipschitz, because the map $F(x, -)$ is an $(s_1,r)$-homotopy and  $H(z,-)$ is $r$-Lipschitz; For any $z \in U_i$ and $j_1, j_2 \in [m]$
\begin{align*}
d_Y(K(z,j_1), K(z,j_2)) &= d_Y (F(H(z, j_1), j_1), F(H(z, j_2), j_2))\\
& \le d_Y(F(H(z, j_1), j_1), F(H(z, j_1), j_2)) \\
& \quad  + d_Y(F(H(z, j_1), j_2), F(H(z, j_2), j_2))\\
 & \le r |j_2 - j_1| + d_Y(F(H(z, j_1), j_2), F(H(z, j_2), j_2))\\
 & \le  r |j_2 - j_1| + s_1 d_X (H(z, j_1), H(z, j_2))\\
 & \le  r |j_2 - j_1| + s_1 r|j_2 -j_1|\\
 & = (s_1  +1)r |j_2 -j_1|.
\end{align*}
Therefore $K$ is an $(s_1s_2, (s_1 +1)r)$-homotopy between the restrictions $f|_{U_i}$ and $g|_{U_i}$.
Since $D_{(s_1 +1)r} (f,g)$ is the minimum number of such covers, $D_{(s_1+1) r}(f \circ h, g \circ h') \leq D_r(h, h')$.
\end{proof}

In Proposition \ref{pC.3}, if the homotopy has smaller scale $s_1$, the scale of distance becomes smaller proven in Proposition \ref{pC.3n}.

\begin{proposition}\label{pC.3n}
Let $X$, $Y$, and $Z$ be metric spaces, $s_1, s_2, r \ge 0$, and  $f, g: X \to Y$ be $(s_1, r)$-homotopic maps. Suppose that $h, h': Z \to X$ be any $s_2$-Lipschitz maps for $s_2 \ge 0$. If $s_1 \le 1$, then
\[
D_{r}(f \circ h, g \circ h') \leq D_r(h, h').
\]
\end{proposition}

\begin{proof}
Assume that $D_r(h, h') = k$. Then there exists an open cover $\{U_0, \dots, U_k\}$ of $Z$ such that $h|_{U_i} \simeq_{(s_2, r)} h'|_{U_i}$ for each $i = 0, \dots, k$. 
Then since $f$ is $s_1$-Lipschitz and $s_1 \le 1$, by Lemma \ref{pC.0} (1),
$f \circ h|_{U_i} \simeq_{(s_2, r)} f\circ h'|_{U_i}$.
Also since $f \simeq_{(s_1,r)} g$, $s_1 \le 1$ and $h'$ is an $s_2$-Lipschitz map, by Lemma \ref{pC.0} (2),
$f \circ h'|_{U_i} \simeq_{(s_2, r)} g \circ h'|_{U_i}$. 
Now we use Proposition \ref{pr2.2}, and achieve the homotopy $
f \circ h|_{U_i} \simeq_{(s_2, r)} g \circ h'|_{U_i}$.
Since $D_r(f \circ h, g \circ h')$ is the smallest integer corresponded to which there exists such a cover, $D_r (f \circ h, g \circ h') \le D_r(h,h')$.
\end{proof}

\section{Homotopy invariance} 
In this section, we demonstrate that discrete homotopy is invariant under certain conditions. Let $X$, $Y$, $X'$ and $Y' $ be topological spaces, and let $f, g : X \to Y$ be maps.  Also, let $\alpha : Y \to Y'$ and $\beta: X' \to X$ be maps such that $\alpha$ has a left homotopy inverse and $\beta$ has a right homotopy inverse . Then $D(\alpha \circ f, \alpha \circ g) = D( f, g)$ and $D(f, g)= D(f \circ \beta, g \circ \beta)$
\cite[Proposition 3.11 and Proposition 3.12]{Mac'ias}. We present Proposition \ref{pHi.1} and \ref{pHi.2} respectively as the modified version  of these results. Furthermore, if there exist homotopy equivalences  $\alpha$ and $\beta$ such that $ \alpha \circ f \circ \beta \simeq f'$ and $ \alpha \circ g \circ \beta \simeq g'$, then $D(f, g) = D( f', g' )$, as shown in \cite[Proposition 3.13]{Mac'ias}. Proposition \ref{pHi.3} provides the corresponding result for the discrete case.

\begin{proposition}\label{pHi.1} 
Let $X$, $Y$, and $Y' $ be metric spaces and let $f, g: X \to Y$ be $s_1$-Lipschitz maps. Also, let $\alpha: Y \to Y'$ be an $s_2$-Lipschitz map  with discrete left homotopy inverse. Then, $D_r(f, g) = D_{s_2r}(\alpha \circ f, \alpha \circ g)$. 
\end{proposition}

\begin{proof}
According to Proposition \ref{pC.1}, $D_{s_2r}(\alpha \circ f, \alpha \circ g) \leq D_r(f, g)$.
For the reverse inequality, let $\beta: Y' \to Y$ be the $1/s_2$-Lipschitz map with $\beta \circ \alpha \simeq_{(1,r)} id_Y$ and also let  $D_{s_2r}(\alpha \circ f,\alpha \circ g) = k$. Therefore, there exist $k+1$ open subsets $U_i \subset X$ such that $(\alpha \circ f) |_{U_i} \simeq_{(s_1s_2,s_2r)} (\alpha \circ g)|_{U_i}$.  Then Proposition \ref{pr2.2} implies that $\beta \circ (\alpha \circ f|_{U_i})   \simeq_{(s_1,r)}  \beta \circ (\alpha \circ g|_{U_i})$ and then
\begin{align*} 
f|_{U_i} = id \circ f|_{U_i} & = (\beta \circ \alpha) \circ f|_{U_i}  \\
& =  \beta \circ (\alpha \circ f|_{U_i})  \\
& \simeq_{(s_1,r)}  \beta \circ (\alpha \circ g|_{U_i})\\
 &= (\beta \circ \alpha) \circ g |_{U_i} 
 \\
 & ={id} \circ g |_{U_i}= g|_{U_i}.
\end{align*}
Therefore, $D_r(f, g) \leq k = D_{s_2r}(\alpha \circ f, \alpha \circ g).$
\end{proof}

\begin{proposition}\label{pHi.2} 
Let $X$, $Y$, and $X'$ be metric spaces and let $f,g: X \to Y$ be two $s_1$-Lipschitz maps. Let $\beta: X' \to X$ be $s_2$-Lipschitz map with a discrete right homotopy inverse. Then, $D_r(f, g)= D_r(f \circ \beta, g \circ \beta)$. 
\end{proposition}

\begin{proof} 
By Proposition \ref{pC.2}, $D_r(f \circ \beta, g \circ \beta) \leq D_r(f,g)$. Assume that $ D_r(f \circ \beta, g \circ \beta)=k$. Then there are $k+1$ open subsets $U_i\subset X'$, covering $X'$ such that $ (f \circ \beta) |_{U_i} \simeq_{(s_1s_2, r)} (g \circ \beta)|_{U_i}$. Let $\eta:X \to X'$ be an $1/s_2$-Lipschitz map being a  right homotopy inverse of $\beta$. Put $V_i = \eta^{-1} (U_i)$, $i =0, \ldots, k$. Since $U_i$'s cover $X'$, their inverse images $V_i$'s cover $X$ and also $\eta (V_i) \subseteq U_i$.
Moreover,  $f|_{V_i} \simeq_{(s_1,r)} g|_{V_i}$ holds as seen below
\begin{align*} 
 (f \circ \beta )|_{U_i} \simeq_{(s_1s_2, r)} (g \circ \beta) |_{U_i} & => (f \circ \beta )|_{U_i} \circ \eta|_{V_i} \simeq_{(s_1s_2\frac{1}{s_2}, r)} (g \circ \beta)|_{U_i} \circ \eta|_{V_i} \\ 
&=> f |_{U_i} \circ (\beta \circ \eta)|_{V_i} \simeq_{(s_1, r)} g|_{U_i} \circ (\beta \circ \eta)|_{V_i} \\
 & => f  \circ id|_{V_i} \simeq_{(s_1, r)} g \circ id|_{V_i} \\
 &=> f|_{V_i} \simeq_{(s_1, r)} g|_{V_i}. 
\end{align*} 
Since $D_r(f,g)$ is the smallest number of such covers, $D_r(f, g) \leq D_r(\alpha \circ f, \alpha \circ g)$. 
\end{proof} 

\begin{proposition}\label{pHi.3} 
Let $X$, $Y$, $X'$ and $Y'$ be metric spaces and  $f,g: X \to Y$ be two $s_1$-Lipschitz maps. Suppose that $\alpha:Y \to Y' $ and $\beta: X' \to X$ are $1$-Lipschitz maps with discrete left homotopy inverse and  discrete right homotopy inverse, respectively, such that $\alpha \circ f \circ \beta \simeq_{(s_1,r)} f'$ and $\alpha \circ g \circ \beta \simeq_{(s_1,r)} g'$. Then, $D_r(f, g) = D_r(f', g')$. 
\[\begin{xy} 
\xymatrix{ 
& X \ar[r]<2pt>^{f} 
\ar[r]<-2pt>_{g} 
&Y \ar[d]^{\alpha} \\ 
& X' \ar[r]<2pt>^ {f'} 
\ar[r]<-2pt>_{g'} 
\ar[u]_{\beta} 
& Y'} 
\end{xy}\] 
\end{proposition} 
\begin{proof} 
Since $\alpha \circ g \circ \beta \simeq_{(s_1,r)} g'$ and $\alpha \circ f \circ \beta \simeq_{(s_1,r)} f'$, by Remark \ref{pD.1}(3) we have $D_r(\alpha \circ f \circ \beta, \alpha \circ g \circ \beta) = D_r(f', g')$. 
By Proposition \ref{pHi.1}, since $\alpha$ is $1$-Lipschitz, $D_r(\alpha \circ f \circ \beta,\alpha \circ g \circ \beta) = D_r( f \circ \beta, g \circ \beta)$ 
and since $\beta$ is $1$-Lipschitz again, by Proposition \ref{pHi.2}, $D_r( f \circ \beta, g \circ \beta) =D_r(f, g)$. Thus
\[D_r(f', g')=D_r(\alpha \circ f \circ \beta,\alpha \circ g \circ \beta)=D_r( f \circ \beta, g \circ \beta) =D_r(f, g). \]
\end{proof}

\end{document}